\documentclass[11pt, a4paper]{amsart}
\usepackage{amssymb, array, amsmath, amscd, pdfpages, enumerate, amsthm, mathtools, overpic}
\usepackage[T1]{fontenc}
\usepackage[colorlinks=true,allcolors=magenta]{hyperref}
\mathtoolsset{showonlyrefs=true}

\DeclareMathOperator*{\Ker}{Ker}
\DeclareMathOperator*{\R}{Re}
\DeclareMathOperator*{\diag}{diag}

\newcommand{\inv}{^{-1}}

\renewcommand{\>}{\rangle}
\newcommand{\dd}{\mathrm{d}}
\newcommand{\B}{\mathcal{B}}

\newcommand{\RR}{\mathbb{R}}
\newcommand{\C}{\mathbb{C}}

\newcommand{\T}{(T(t))_{t\ge 0}}
\newcommand{\ep}{\varepsilon}

\newcommand{\gb}{\beta}
\newcommand{\gl}{\lambda}

\newcommand{\norm}[1]{\|#1\|}
\newcommand{\abs}[1]{|#1|}

\newtheorem{thm}{Theorem}[section]
\newtheorem{prp}[thm]{Proposition}
\newtheorem{lem}[thm]{Lemma}

\theoremstyle{definition}
\newtheorem{rem}[thm]{Remark}

\numberwithin{equation}{section}

\begin{document}

\title[Polynomial stability of a wave-heat network]{Polynomial stability of a \\coupled wave-heat network}

\author[L.~Paunonen]{Lassi Paunonen}
\address[L.~Paunonen]{Mathematics Research Centre, Tampere University, P.O.~ Box 692, 33101 Tampere, Finland}
 \email{lassi.paunonen@tuni.fi}

\author[D. Seifert]{David Seifert}
\address[D. Seifert]{School of Mathematics, Statistics and Physics, Newcastle University, Herschel Building, Newcastle upon Tyne, NE1 7RU, UK}
\email{david.seifert@ncl.ac.uk}

\begin{abstract}
We study the long-time asymptotic behaviour of a topologically non-trivial network of wave and heat equations. By analysing  the simpler wave and the heat networks separately, and then applying  recent results for abstract coupled systems, we establish energy decay at the rate $t^{-4}$ as $t\to\infty$ for all classical solutions.
\end{abstract}


 \thanks{The authors gratefully acknowledge financial support from the Research Council of Finland grant number 349002  and COST Action CA18232.}

\maketitle

\section{Introduction}\label{sec:intro}

In this paper we study the long-time asymptotic behaviour of  the system of coupled one-dimensional wave and heat equations shown in~Figure~\ref{fig:WHnetworkMultiHeat}. 

\begin{figure}[h!]
\begin{center}
\includegraphics[width=0.75\linewidth]{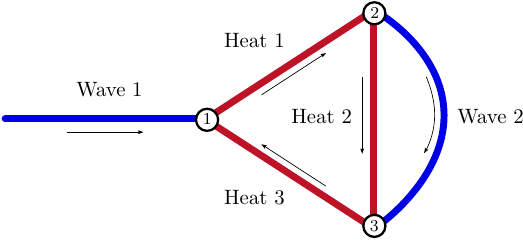}
\end{center}
\caption{A wave-heat network with heat parts (red) and wave parts (blue).
The arrows indicate the directions in which the spatial coordinates increase.
}
\label{fig:WHnetworkMultiHeat}
\end{figure}

Along each of the blue edges we solve the wave equation $y_{tt}^k(x,t)=y_{xx}^k(x,t)$, for $x\in(0,1)$, $t>0$ and $k=1,2$, and along each of the red edges we solve the heat equation $w_{t}^k(x,t)=\beta _kw_{xx}^k(x,t)$, for $x\in(0,1)$, $t>0$ and $k=1,2,3$, where $\beta_1,\beta_2,\beta_3>0$ are  constants. We impose the boundary condition $ y_t^1(0,t)=0$, together with the Kirchhoff-type coupling conditions 
$$
\begin{cases}
 y_t^1(1,t)=w^1(0,t) = w^3(1,t)  \\
 y_x^1(1,t)-\beta_1 w_x^1(0,t) + \beta_3 w_x^3(1,t) =0\\
\end{cases}
$$
at vertex 1,
$$\begin{cases}
w^1(1,t) = w^2(0,t) =  y_t^2(0,t),  \\
\beta_1 w_x^1(1,t) - \beta_2 w_x^2(0,t) - y_x^2(0,t) =0\\
\end{cases}
$$
at vertex 2, and 
$$
\begin{cases}
w^2(1,t) = w^3(0,t) =  y_t^2(1,t),  \\
\beta_2 w_x^2(1,t) - \beta_3 w_x^3(0,t) + y_x^2(1,t)=0 \\
\end{cases}
$$
at vertex 3, in all cases for $t>0$. We study the total energy of the wave-heat system, defined as
$$E(t)=\frac12\sum_{k=1}^2\int_0^1|y^k_x(x,t)|^2+|y^k_t(x,t)|^2\,\dd x+\frac12\sum_{k=1}^3\int_0^1|w^k(x,t)|^2\,\dd x$$ 
for $t\ge0$, with the aim of establishing decay of this energy as $t\to\infty$  for initial data leading to either a mild or a  classical solution of the corresponding abstract Cauchy problem. The following is our main result.

\begin{thm}
\label{thm:WaveMultiHeat}
The energy of every solution of the wave-heat system satisfies $E(t)\to0$ as $t\to\infty$, and for classical solutions we have $E(t)=o(t^{-4})$ as $t\to\infty$.
\end{thm}

Networks of PDEs have been studied extensively in the literature. Existence of solutions has been established  for instance in~\cite{JacKai19,Klo12,KraMug21,MatSik07}, while polynomial stability of networks of  wave and heat equations has been studied in~\cite{LiWan24, Ng20b} for certain particular network configurations and  in~\cite{HanZua16} for general star-shaped networks. 
Stability of various other types of PDEs on networks
has been considered in~\cite{Amm07,AmmJel04,AmmShe18,DagZua06book,EggKug18,MatSik07,ValZua09}, to give just a small selection. This work also relates to the study of fluid-structure interactions and wave-heat systems on spatial domains of one- and higher-dimensional spatial domains; see for instance~\cite{AvaLas16, BatPau16, BatPau19, DelPau23,Duy07, ZhaZua04,ZhaZua07}.

In the present paper, we establish well-posedness and stability of our wave-heat system by following the approach taken in~\cite{NiPaSe24} for composite systems subject  to abstract boundary couplings. 
The corresponding results were already used in~\cite[Sec.~4]{NiPaSe24} to investigate stability of certain types of networks of wave and heat equations. 
In this paper we investigate the structure depicted in Figure~\ref{fig:WHnetworkMultiHeat}, which is more complicated than the PDE networks considered in~\cite[Sec.~4]{NiPaSe24} on account of the fact that there are now multiple vertices at which wave equations are coupled with heat equations.  We demonstrate 
how the abstract results in~\cite{NiPaSe24} can be applied, in a slightly more intricate way than was necessary in~\cite[Sec.~4]{NiPaSe24}, to analyse this more complicated model.

The method presented in~\cite{NiPaSe24} is based on a decomposition of the wave-heat system shown in Figure~\ref{fig:WHnetworkMultiHeat} into two smaller networks; in our situation, 
 one  ``wave network'' and one ``heat network''. The decomposition of the full network model and the interaction between the two subnetworks is illustrated in Figure~\ref{fig:WHnetworkDecomposition}. We analyse these two simpler networks separately in Section~\ref{sec:WHNetwork} in order to prove Theorem~\ref{thm:WaveMultiHeat} in Section~\ref{sec:energy}. First, in Section~\ref{sec:Prelim}, we reproduce the theoretical results from~\cite{NiPaSe24} which will form the basis for our analysis.

\begin{figure}[h!]
\begin{center}
\includegraphics[width=0.8\linewidth]{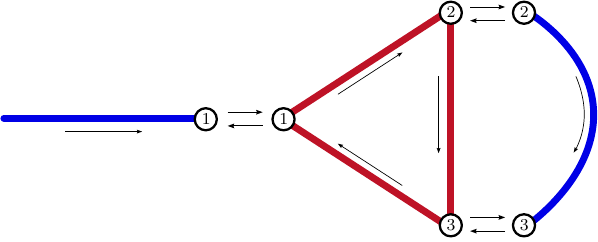}
\end{center}
\caption{Decomposition of the network into a wave network (the two blue edges) and a heat network (the three red edges) which interact via coupling at the three vertices.
}
\label{fig:WHnetworkDecomposition}
\end{figure}

 We use standard notation and terminology throughout. All normed spaces will be considered over the complex field. Given normed spaces $X$ and $Y$, we write $\B(X,Y)$ for the space of bounded linear operators from $X$ to $Y$, and we write $\B(X)$ for $\B(X,X)$. We denote the domain of a linear operator $L$  by $D(L)$, and we endow this space with the graph norm. We denote the  kernel of a linear operator $L$ by  $\Ker L$, and we write  $\rho(L)$ for the resolvent set of $L$. Furthermore, given a bounded linear operator $T\in\B(X)$ on a Hilbert space $X$, we denote by $\R T$ the self-adjoint operator $\frac12(T+T^*)$, and given two self-adjoint bounded linear operators $S,T\in\B(X)$ on a Hilbert space $X$ we write $T\ge S$ to mean that $\langle Tx-Sx,x\rangle\ge0$ for all $x\in X$. If $p$ and $q$ are two real-valued quantities we write $p\lesssim q$ to express that $p\le Cq$ for some constant $C>0$ which is independent of all parameters that are free to vary in a given situation. We shall also make use of standard `big-O' and `little-o' notation.

\section{Boundary nodes and an abstract stability result}
\label{sec:Prelim}

In this section we present the abstract setting we shall use to analyse our wave-heat network, and we recall the main result of~\cite{NiPaSe24}. Let  $X$ and $U$ be Hilbert spaces, and let $L\colon D(L)\subseteq X\to X$, $G, K\colon D(L)\subseteq X\to U$ be linear operators. The triple $(G,L,K)$
is said to be a \emph{boundary node} on  $(U,X,U)$ if 
\begin{itemize}
\item[\textup{(i)}]  $G,K\in\B(D(L),U)$;
\item[\textup{(ii)}] the restriction $A $ of $L$ to $\Ker G$ generates a $C_0$-semigroup on $X$;
\item[\textup{(iii)}]  the operator $G$ has a bounded right-inverse, i.e.\ there exists $G^r\in\B(U,D(L))$ such that $GG^r=I$.
\end{itemize}
The boundary node $(G,L,K)$ is said to be \emph{impedance passive} if
\begin{equation}\label{eq:pass}
\R\<Lx,x\>_X\le\R\<Gx,Kx\>_U,\qquad x\in D(L),
\end{equation}
and in this case  the $C_0$-semigroup generated by $A$ is necessarily contractive, by the Lumer--Phillips theorem \cite[Thm.~II.3.15]{EngNag00book}.

If $(G,L,K)$ is a boundary node on $(U,X,U)$ and  $A$ denotes the restriction of $L$ to $\Ker G$, we define the \emph{transfer function}  $P\colon\rho(A)\to\B(U)$  by $P(\lambda)u=Kx$ for all $\lambda\in\rho(A)$ and $u\in U$, where $x\in D(L)$ is the unique solution of the abstract elliptic problem
$$(\lambda-L)x=0,\qquad Gx=u.$$
Existence of $P$ follows from~\cite[Rem.~10.1.5 \& Prop.~10.1.2]{TucWei09book}. Recall also from~\cite[Prop.~2.4(b)]{NiPaSe24} that if the boundary node $(G,L,K)$  is impedance passive then $\R P(\gl)\ge0$ for all $\gl\in\rho(A)\cap\overline{\C_+}$.
As in~\cite{NiPaSe24}, we do not distinguish notationally between  $P$ and its analytic extensions to larger domains. For more information on boundary nodes, see  \cite[Ch.~11]{JacZwa12book}, \cite{MalSta07}.

Let $(G_1, L_1, K_1)$ and $( G_2, L_2,  K_2)$ be two impedance passive boundary nodes on $(U,X_1,U)$ and $(U,X_2,U)$, respectively, where $X_1,$ $X_2$ and $U$ are Hilbert spaces, and consider the coupled system
\begin{equation}\label{eq:coupled_sys}
\left\{\begin{aligned}
	  \dot{z}_1(t) &= L_1 z_1(t),  \qquad&t\ge0,\\
	  	  \dot{z}_2(t) &= L_2 z_2(t),  \qquad&t\ge0,\\
	  G_1 z_1(t) &=K_2z_2(t), &t\ge0,\\
	  G_2z_2(t) &= -K_1 z_1(t), &t\ge0,\\
	  z_1(0)&\in X_1,\ z_2(0)\in X_2.
	\end{aligned}\right.
\end{equation}
We may reformulate this coupled system as an abstract Cauchy problem
\begin{equation*}
\left\{\begin{aligned}
	  \dot{z}(t) &= A z(t),  \qquad t\ge0,\\
	  	  z(0) &=z_0  
	\end{aligned}\right.
\end{equation*}
for $z(\cdot)=(z_1(\cdot),z_2(\cdot))^\top$ on $X=X_1\times X_2$, where $z_0=(z_1(0),z_2(0))^\top\in X$  and $A=\diag(L_1,L_2)\colon D(A)\subseteq X\to X$ with 
$$D(A)=\big\{(x_1,x_2)^\top\in D(L_1)\times D(L_2): G_1x_1=K_2x_2,\ G_2x_2=-K_1x_1\big\}.$$
For $j=1,2$ we denote by $A_j$ the restriction of $L_j$ to $\Ker G_j$ and we denote the transfer function of the boundary node $(G_j,L_j,K_j)$  by $P_j$.

The following is a special case of~\cite[Thm.~3.1]{NiPaSe24}. The result makes it possible to deduce a growth bound for the resolvent of $A$ from information about the boundary nodes $(G_1+K_1,L_1, K_1)$ and $(G_2, L_2, K_2)$. Here we may think of $(G_1+K_1,L_1, K_1)$ as a boundary-damped version of the boundary node $(G_1,L_1, K_1)$. 
For instance, in our wave-heat network 
 $(G_1+K_1,L_1, K_1)$ describes a wave network with dissipative boundary conditions.
Note that in the coupled system this boundary-damping is replaced by the interconnection with the boundary node $(G_2,L_2,  K_2)$.

\begin{thm}\label{thm:res}
Let $(G_1, L_1, K_1)$ and $(G_2,L_2,  K_2)$ be two impedance passive boundary nodes on $(U,X_1,U)$ and $(U,X_2,U)$, respectively, and let $A_0$ denote the restriction of $L_1$ to $\Ker (G_1+K_1)$. Suppose there exists a non-empty set $E\subseteq\{s\in\RR:is\in \rho(A_0)\cap\rho(A_2)\}$  and that
\begin{equation}\label{eq:res_bd}
\sup_{s\in E}\|(is-A_j)\inv\|<\infty
\end{equation}
for  $j=0,2$. Suppose furthermore that there exists a function $\eta\colon E\to(0,\infty)$ such that 
$\R P_2(is)\ge\eta(s)I$ for all $s\in E.$
Then $A$ generates a contraction semigroup on $X$, $iE\subseteq\rho(A)$ and
\begin{equation}\label{eq:res_est}
\|(is-A)\inv\|\lesssim  \frac{\mu(s)}{\eta(s)},\qquad s\in E,
\end{equation}
where $\mu(s)=1+\|P_2(1+is)\|^2$ for  $s\in \RR$.
\end{thm}

\begin{rem}
Note that~\cite[Thm.~3.1]{NiPaSe24} is in fact more general than this. In particular, in~\cite[Thm.~3.1]{NiPaSe24} the space $U$ is allowed to be different for each of the two boundary nodes, and in addition we may replace~\eqref{eq:res_bd} by the weaker condition that the norms of the two resolvents can be estimated from above along the imaginary axis by functions which need not be bounded. 
\end{rem}

\section{Analysis of the wave-heat network}
\label{sec:WHNetwork}

We begin by expressing the dynamics of the wave-heat network as an abstract Cauchy problem on a Hilbert space $X = L^2(0,1)^7$.
To this end, we define
$$
z(t)= (y_x^1(\cdot,t),y_t^1(\cdot,t),
y_x^2(\cdot,t),y_t^2(\cdot,t),
w^1(\cdot,t),w^2(\cdot,t),
w^3(\cdot,t)
)^\top
$$
for $t\geq 0$.
We note that the total energy of the solutions of the wave-heat network satisfies
$$
E(t) = \frac{1}{2} \norm{z(t)}^2, \qquad t\geq 0.
$$
The evolution of the state $z(\cdot)$  is described by the abstract Cauchy problem
\begin{equation}
\label{eq:ACP}
\begin{cases}
\dot z(t) = Az(t), \qquad t\geq0,\\
z(0) = z_0\in X,
\end{cases}
\end{equation}
where the generator $A\colon D(A)\subseteq X\to X$ is defined by
$$
Ax = (v_1',u_1',v_2',u_2',\gb_1 w_1'',\gb_2 w_2'',\gb_3 w_3'')^\top
$$
for $x=(u_1,v_1,u_2,v_2,w_1,w_2,w_3)^\top$
in the domain
$$
\begin{aligned}
D(A)
&=\bigl\{(u_1,v_1,u_2,v_2,w_1,w_2,w_3)^\top\in H^1(0,1)^4\times H^2(0,1)^3:
v_1(0)=0, \\
 &\hspace{1cm} v_1(1)=w_1(0) = w_3(1) , \quad
 u_1(1)-\beta_1 w_1'(0) + \beta_3 w_3'(1) =0,\\
 &\hspace{1cm} w_1(1) = w_2(0) =  v_2(0) , \quad 
\beta_1 w_1'(1) - \beta_2 w_2'(0) - u_2(0) =0,\\
 &\hspace{1cm} w_2(1) = w_3(0) =  v_2(1)  , \quad
\beta_2 w_2'(1) - \beta_3 w_3'(0) + u_2(1)=0 
~\bigl\}.
\end{aligned}$$
As will become clear in Section~\ref{sec:energy} the operator $A$ generates a $C_0$-semigroup $\T$ of contractions on $X$. We recall that the \emph{mild solution} of the abstract Cauchy problem~\eqref{eq:ACP} is given by $z(t)=T(t)z_0$ for $t\ge0$, and that $z(\cdot)$ is a \emph{classical solution} if and only if $z_0\in D(A)$; see~\cite[Sec.~II.6]{EngNag00book}. 

In order to apply the results in Section~\ref{sec:Prelim}, we  express $A$ and $D(A)$ in terms of two boundary nodes $(G_1,L_1,K_1)$ and $(G_2,L_2,K_2)$ on $(\C^3,L^2(0,1)^4,\C^3)$ and $(\C^3,L^2(0,1)^3,\C^3)$, respectively. 
We achieve this by decomposing the state $z(\cdot)$ into two parts, one related to the two wave equations, the second related to the three heat equations. More specifically, for $t\ge0$ we write
$$
z(t) = \begin{pmatrix}
z_1(t)\\ z_2(t)
\end{pmatrix}, \qquad
z_1(t)= \begin{pmatrix}u_1(t)\\v_1(t)\\
u_2(t)\\v_2(t)\end{pmatrix},
\qquad 
z_2(t)=\begin{pmatrix}
w_1(t)\\ w_2(t)\\ w_3(t) \end{pmatrix}.
$$
With respect to this decomposition the operator $A$ can be  expressed in the form
$A=\diag(L_1,L_2)$, where  
$$
L_1 \begin{pmatrix}u_1\\v_1\\u_2\\v_2\end{pmatrix} = \begin{pmatrix}v_1'\\u_1'\\v_2'\\u_2'\end{pmatrix}$$
with domain
$D(L_1) = \{ (u_1,v_1,u_2,v_2)^\top\in H^1(0,1)^4 : v_1(0)=0 \}$
and
$$L_2 \begin{pmatrix}w_1\\w_2\\w_3\end{pmatrix} = \begin{pmatrix}\gb_1 w_1''\\\gb_2 w_2''\\ \gb_3 w_3''\end{pmatrix}$$
with domain $D(L_2) =\{ (w_1,w_2,w_3)^\top\in H^2(0,1)^3 : w_3(1)=w_1(0), ~ w_1(1)=w_2(0), w_2(1)=w_3(0)\}$.
In addition, we define $G_1,K_1\colon D(L_1)\to\C^3$ by 
$$G_1\begin{pmatrix}u_1\\v_1\\u_2\\v_2\end{pmatrix} = \begin{pmatrix}-u_1(1)\\u_2(0)\\-u_2(1)\end{pmatrix}
, \qquad K_1 \begin{pmatrix}u_1\\v_1\\u_2\\v_2\end{pmatrix} = \begin{pmatrix}-v_1(1)\\-v_2(0)\\-v_2(1)\end{pmatrix}$$
and  $G_2,K_2\colon D(L_2)\to\C^3$ by
$$G_2\begin{pmatrix}w_1\\w_2\\w_3\end{pmatrix} = \begin{pmatrix}w_1(0)\\w_2(0)\\w_3(0)\end{pmatrix}
, \qquad K_2\begin{pmatrix}w_1\\w_2\\w_3\end{pmatrix} = \begin{pmatrix}\gb_3 w_3'(1)-\gb_1 w_1'(0)\\\gb_1 w_1'(1)-\gb_2 w_2'(0)\\\gb_2 w_2'(1)-\gb_3 w_3'(0)\end{pmatrix},$$
noting that these definitions ensure that $D(A)=\{(x_1,x_2)^\top\in D(L_1)\times D(L_2): G_1x_1=K_2x_2,\ G_2x_2=-K_1x_1\},$ as required in Section~\ref{sec:Prelim}.

We now show that $(G_1,L_1,K_1)$ and $(G_2,L_2,K_2)$ are impedance passive boundary nodes, and that the restrictions $A_0$ and $A_2$ of $L_1$ to $\Ker(G_1+K_1)$ and of $L_2$ to $\Ker G_2$, respectively, have uniformly bounded resolvents along the imaginary axis. Furthermore, we derive bounds for the transfer function $P_2$ of the second node. Once these objectives have been met, Theorem~\ref{thm:WaveMultiHeat} will be a straightforward consequence of Theorem~\ref{thm:res} and the stability theory of strongly continuous semigroups.

\subsection{Analysis of the wave network}
\label{sec:WaveNetwork}

In this section we analyse the wave network and show that $(G_1,L_1,K_1)$ defines an impedance passive boundary node on $(\C^3,X_1,\C^3)$, where $X_1=L^2(0,1)^4$. From the above definitions it is clear that $G_1,K_1\in \B(D(L_1),\C^3)$, and since $G_1$ is surjective and $\C^3$ is finite-dimensional we know that $G_1$ has a bounded right-inverse.
We note that the restriction $A_1$ of $L_1$ to $\Ker G_1$ has the structure $A_1 = \diag(A_{1,1},A_{1,2})$, where $A_{1,1},A_{1,2}\colon(u,v)^\top\mapsto(v',u')^\top$ with respective domains $D(A_{1,1}) = \{(u,v)^\top\in H^1(0,1)^2:u(1)=0, v(0)=0\}$ and
$D(A_{1,2}) = H_0^1(0,1)\times H^1(0,1)$.
Since both $A_{1,1}$ and $A_{1,2}$ generate strongly continuous contraction semigroups on $L^2(0,1)^2$, $A_1$ generates a contraction semigroup on $X_1$.
Thus $(G_1,L_1,K_1)$ is a boundary node.  A straightforward computation using integration by parts shows that $\R\langle L_1 x,x\rangle=\R\langle G_1x,K_1x\rangle$ for all $x\in D(L_1)$, so the boundary node $(G_1,L_1,K_1)$ is impedance passive. Arguing as for instance in~\cite[Ex.~9.2.1]{JacZwa12book} for the wave equations along each of the blue edges in Figure~\ref{fig:WHnetworkDecomposition} it is possible to show that the semigroup generated by the restriction $A_0$ of $L_1$ to $\Ker(G_1+K_1)$ is uniformly exponentially stable. In fact, by adapting the argument given in~\cite[Sec.~1]{CoxZua95} we see that this semigroup is even nilpotent in our case. Either way, it follows that  $i\RR\subseteq\rho(A_0)$ and $\sup_{s\in\RR}\|(is-A_0)\inv\|<\infty$.

\subsection{Analysis of the heat network}
\label{sec:HeatNetwork}

It remains to analyse the heat network. In this section we show that $(G_2,L_2,K_2)$ defines an impedance passive boundary node on $(\C^3,X_2,\C^3)$, where $X_2=L^2(0,1)^3$, and we analyse the transfer function $P_2$ of this boundary node.
It would certainly be possible to establish the boundary node property and stability of $(G_2,L_2,K_2)$ directly from the definitions of $L_2$, $G_2$, and $K_2$. Instead, however, we demonstrate how the analysis can be simplified by viewing the boundary node as an \emph{interconnection} of three individual heat equations. This approach allows us to use the interconnection theory developed in~\cite{AalMal13} to show the boundary node property and impedance passivity of $(G_2,L_2,K_2)$, and also provides an efficient way of deriving a formula for the transfer function $P_2$.
We begin, therefore, by analysing a single heat equation with two inputs and two outputs. In what follows we use the principal branch of the logarithm to define the complex square root. 

\begin{lem}\label{lem:SingleHeatProps}
Let $X=L^2(0,1)$ and $\gb>0$, and define $L\colon D(L)\subseteq X\to X$ and $G,K\colon D(L)\to \C^2$ by $D(L)=H^2(0,1)$ and 
$$
L w = \gb w'', \qquad G w = \begin{pmatrix}w(0)\\w(1)\end{pmatrix}, \qquad  Kw=\begin{pmatrix}-\gb w'(0)\\\gb w'(1)\end{pmatrix}
$$
for $w\in D(L)$.
Then $(G,L,K)$ is an impedance passive boundary node on $(\C^2,X,\C^2)$,  the semigroup generated by the restriction $A$ of $L$ to $\Ker G$ is uniformly exponentially stable, and the map $w\mapsto (Gw,Kw)^\top$ is a surjection from $D(L)$ onto $\C^4$. Furthermore, the  transfer function $P$ of the boundary node $(G,L,K)$ satisfies
$$
P(\gl) = \begin{pmatrix}p_1(\gl)&p_2(\gl)\\ p_2(\gl)& p_1(\gl)\end{pmatrix},\qquad \gl\in\overline{\C_+},
$$
where
$$
p_1(\gl) 
 = \frac{\sqrt{\gl \gb}}{\tanh\sqrt{\gl/\gb}}, 
\qquad p_2(\gl)
 = -\frac{\sqrt{\gl \gb}}{\sinh\sqrt{\gl/\gb}}
$$
for $\gl\in \overline{\C_+}\setminus \{ 0\}$ and $p_1(0)=-p_2(0)=\gb$.
\end{lem}

\begin{proof}
It is straightforward to establish that $(G,L,K)$ is an impedance passive boundary node on $(\C^2,X,\C^2)$,  that  the semigroup generated by $A$ is uniformly exponentially stable, and that the map $w\mapsto (Gw,Kw)^\top$ is a surjection from $D(L)$ onto $\C^4$. In order to determine the transfer function $P$, let $\gl\in\overline{\C_+}$ and $u=(u_1,u_2)^\top\in\C^2$. Then $P(\gl)u=(-\gb w'(0),\gb w'(1))^\top$, where $w\in H^2(0,1)$ is the unique solution of the problem $\gl w-\gb w''=0$ subject to the boundary conditions $w(0)=u_1$ and $w(1)=u_2$. Solving this boundary-value problem yields
$w(x)=u_1(1-x)+u_2x$ for $x\in(0,1)$ if $\gl=0$, and $$w(x)=u_1\cosh(\nu x)+\left(\frac{u_2}{\sinh \nu}-\frac{u_1}{\tanh \nu}\right)\sinh (\nu x),\qquad x\in(0,1),$$
if $\gl\in\overline{\C_+}\setminus\{0\}$,
where $\nu=\sqrt{\gl/\gb}$. The formula for $P(\gl)$ follows at once.
\end{proof}

The following is the key result in analysing the heat part of our system.

\begin{prp}
\label{prp:HeatNetAnalysis}
The triple $(G_2,L_2,K_2)$ is an impedance passive boundary node on $(\C^3,X_2,\C^3)$, and the semigroup generated by the restriction $A_2$ of $L_2$ to $\Ker G_2$ is uniformly exponentially stable. 
Furthermore, the transfer function $P_2$ satisfies
$$\|P_2(1+is)\|\lesssim 1+\abs{s}^{1/2}, \qquad s\in\RR,$$
and, given any  $\ep>0$, there exists a constant $c_\ep>0$ such that
\begin{equation}\label{eq:Re_lb}
\R P_2(is) \ge c_\ep(1+{\abs{s}^{1/2}})I, \qquad s\in\RR\setminus (-\ep,\ep).
\end{equation}
\end{prp}

\begin{proof}
We begin by considering a boundary node $(G,L,K)$ on $(\C^6,X_2,\C^6)$ defined by three separate heat equations, each of the form in Lemma~\ref{lem:SingleHeatProps}.
More precisely, let $L\colon D(L)\subseteq X_2\to X_2$ and $G,K\colon D(L)\to \C^6$ be defined by 
$$
L \begin{pmatrix}w_1\\w_2\\w_3 \end{pmatrix}= \begin{pmatrix}\gb_1 w_1''\\\gb_2 w_2''\\\gb_3 w_3''\end{pmatrix}, \quad G \begin{pmatrix}w_1\\w_2\\w_3\end{pmatrix} = \begin{pmatrix}w_1(0)\\w_1(1)\\w_2(0)\\w_2(1)\\w_3(0)\\w_3(1)\end{pmatrix}, \quad
K \begin{pmatrix}w_1\\w_2\\w_3\end{pmatrix} = 
\begin{pmatrix}-\gb_1 w_1'(0)\\\gb_1 w_1'(1)\\-\gb_2 w_2'(0)\\\gb_2 w_2'(1)\\-\gb_3 w_3'(0)\\\gb_3 w_3'(1)\end{pmatrix},
$$
respectively, for $(w_1,w_2,w_3)^\top\in D(L)=H^2(0,1)^3$. We note that we may recover our original heat network from the decoupled heat equations
 by choosing  three inputs $u_1,u_2,u_3$ and three outputs $y_1,y_2,y_3$  in a suitable way, namely $w_1(0)=w_3(1)=u_1$ and $y_1 = \gb_3 w_3'(1)-\gb_1 w_1'(0)$, $w_1(1)=w_2(0)=u_2$ and $y_2 = \gb_1 w_1'(1)-\gb_2 w_2'(0)$, and finally $w_2(1)=w_3(0)=u_3$ and $y_3 = \gb_2 w_2'(1)-\gb_3 w_3'(0)$.
Since these interconnections fit into the framework of~\cite{AalMal13}, the properties in Lemma~\ref{lem:SingleHeatProps} and~\cite[Thm.~3.3]{AalMal13} together with~\cite[Rem.~2.2]{NiPaSe24} imply that
 $(G_2,L_2,K_2)$ is an impedance passive boundary node on $(\C^3,X_2,\C^3)$. Furthermore, using the diagonal structure of the restriction $A_2$ of $L_2$ to $\Ker G_2$ it is easy to see that the semigroup generated by $A_2$ is uniformly exponentially stable.

It remains to establish bounds for the transfer function $P_2$ of the boundary node  $(G_2,L_2,K_2)$. Let $\gl\in\rho(A_2)$ and $u\in \C^3$. From the definition of the transfer function we have $P_2(\gl)u=K_2x,$ where $x\in D(L_2)$ is the unique solution of the abstract elliptic problem $(\gl-L_2)x=0$ subject to the boundary condition $G_2x=u$. From the properties of the interconnections described  in the previous paragraph it is clear that $x$ solves this problem if and only if it is the unique element of $D(L)$ solving the problem $(\gl-L)x=0$ subject to the boundary condition $Gx=R^\top u$, where
$$R=\begin{pmatrix}
1&0&0&0&0&1\\
0&1&1&0&0&0\\
0&0&0&1&1&0
\end{pmatrix}.
$$
Since $K_2x=R Kx$, we deduce that $P_2(\gl)u=R P(\gl)R^\top u$, and hence $P_2(\gl)=R P(\gl)R^\top$.
Now from Lemma~\ref{lem:SingleHeatProps} and the diagonal structure of $L$ it is clear that $P(\gl)=\diag(P^1(\gl),P^2(\gl),P^3(\gl))$ for $\gl\in\rho(A_2)$, where
$$
P^k(\gl) = \begin{pmatrix}p^k_1(\gl)&p^k_2(\gl)\\ p^k_2(\gl)& p^k_1(\gl)\end{pmatrix},\qquad \gl\in\overline{\C_+},
$$
with
$$
p_1^k(\gl) 
 = \frac{\sqrt{\gl \gb_k}}{\tanh\sqrt{\gl/\gb_k}}, 
\qquad p_2^k(\gl)
 = -\frac{\sqrt{\gl \gb_k}}{\sinh\sqrt{\gl/\gb_k}}
$$
for $\gl\in \overline{\C_+}\setminus \{ 0\}$ and $p_1^k(0)=-p_2^k(0)=\gb_k$ for $k=1,2,3$. A straightforward computation now gives
$$P_2(\gl)=\begin{pmatrix}
p^1_1(\gl)+p^3_1(\gl) &p^1_2(\gl)&p^3_2(\gl)\\
p^1_2(\gl)&p^1_1(\gl) + p^2_1(\gl) & p^2_2(\gl) \\
p^3_2(\gl) & p^2_2(\gl) & p^2_1(\gl) + p^3_1(\gl)
\end{pmatrix},\qquad \gl\in\overline{\C_+}.$$ From the formulas for $p_1^k$ and $p_2^k$ for $k=1,2,3$ it follows easily that $\|P_2(1+is)\|\lesssim1+|s|^{1/2}$ for $s\in\RR$. Finally we establish a lower bound for $\R P_2(is)$ when $s\in\RR\setminus\{0\}$. To this end, let us write $q^k_j(s)$ for $\R p_j^k(is)$, where $j=1,2$, $k=1,2,3$ and $s\in\RR$, so that
 $$\R P_2(is)=\begin{pmatrix}
q^1_1(s)+q^3_1(s) &q^1_2(s)&q^3_2(s)\\
q^1_2(s)&q^1_1(s) + q^2_1(s) & q^2_2(s) \\
q^3_2(s) & q^2_2(s) & q^2_1(s) + q^3_1(s)
\end{pmatrix},\qquad s\in\RR.
$$
Let $a_k(s)=(|s|/(2\beta_k))^{1/2}$ for $k=1,2,3$ and $s\in\RR$. Then 
$$\begin{aligned}
q_1^k(s)&=\gb_k a_k(s)\frac{\cosh a_k(s)\sinh a_k(s)+  \cos a_k(s)\sin a_k(s)}{\sinh^2a_k(s)+\sin^2a_k(s)},\\
q_2^k(s)&=-\gb_k a_k(s)\frac{\cos a_k(s)\sinh a_k(s)+  \cosh a_k(s)\sin a_k(s)}{\sinh^2a_k(s)+\sin^2a_k(s)}
\end{aligned}$$
for $k=1,2,3$ and $s\in\RR\setminus\{0\}$. We thus see that $q_1^k(s)>0$ for $k=1,2,3$ and $s\in\RR\setminus\{0\}$. A standard calculation shows that
$$
q_1^k(s)^2-q_2^k(s)^2
=\frac{\beta_k|s|}{2}\frac{\sinh^2 a_k(s)-\sin^2 a_k(s)}{\sinh^2 a_k(s)+\sin^2 a_k(s)},
$$
and in particular $|q_1^k(s)|>|q_2^k(s)|$ for $k=1,2,3$ and $s\in\RR\setminus\{0\}$. Hence the matrix $\R P_2(is)$ is strictly diagonally dominant, and in particular non-singular,  for all $s\in\RR\setminus\{0\}$. It follows that there exists a continuous map $\eta\colon\RR\setminus\{0\}\to(0,\infty)$ such that $\R P_2(is)\ge\eta(s)I$ for all $s\in\RR\setminus\{0\}$. Since $q_1^k(s)\gtrsim|s|^{1/2}$ as $|s|\to\infty$ while $|q_2^k(s)|\to0$ exponentially fast as $|s|\to\infty$ for $k=1,2,3$, we obtain the  lower bound in~\eqref{eq:Re_lb} for arbitrary $\ep>0$.
\end{proof}

\section{Energy decay}
\label{sec:energy}

In this final section we combine the results presented in 
Sections~\ref{sec:Prelim} and~\ref{sec:WHNetwork} in order to prove our main result.

\begin{proof}[Proof of Theorem~\ref{thm:WaveMultiHeat}]
We consider the abstract Cauchy problem corresponding to our wave-heat system as explained in Section~\ref{sec:WHNetwork}. 
The observations made in Section~\ref{sec:WaveNetwork} together with Proposition~\ref{prp:HeatNetAnalysis} imply that 
the boundary nodes
 $(G_1,L_1,K_1)$ and
 $(G_2,L_2,K_2)$ 
 satisfy
the conditions of Theorem~\ref{thm:res}
with $E=\RR\setminus (-\varepsilon,\varepsilon)$ for any $\varepsilon>0$.
 Therefore
%
the operator $A$ defined in Section~\ref{sec:WHNetwork} 
generates a contraction semigroup $\T$ on the space $X=X_1\times X_2$. Proposition~\ref{prp:HeatNetAnalysis} also tells us that $\|P_2(1+is)\|^2\lesssim 1+|s|$ for $s\in\RR$ and that, given any $\ep>0$, there exists $c_\ep$ such that $\R P_2(is)\ge c_\ep(1+|s|^{1/2})I$ for all $s\in\RR\setminus(-\ep,\ep)$. It follows from Theorem~\ref{thm:res} that $\{is:s\in\RR, |s|>\ep\}\subseteq\rho(A)$ and $\|(is-A)\inv\|\lesssim 1+|s|^{1/2}$ for $|s|>\ep$, for any given $\ep>0$. 
%
%
Next we establish that $0\in\rho(A)$.
We first note that since both $A_1$ and $A_2$ have compact resolvent, it follows from~\cite[Thm.~3.7]{NiPaSe24} 
that  $A$, too, has compact resolvent. It therefore suffices to show that $A$ is injective.
To this end, suppose that $x=(u_1,v_1,u_2,v_2,w_1,w_2,w_3)^\top\in \Ker A$. Then the functions $u_k,v_k$ are constant for $k=1,2$, and the functions $w_k$ are affine for $k=1,2,3$.
Moreover, the coupling conditions in $D(A)$ give 
\begin{align*}
v_1(0)=0, 
 \quad 
& v_1(1)=w_1(0) = w_3(1) , \quad
 u_1(1)-\gb_1 w_1'(0) + \gb_3 w_3'(1) =0,\\
 & w_1(1) = w_2(0) =  v_2(0) , \quad 
\gb_1 w_1'(1) - \gb_2 w_2'(0) - u_2(0) =0,\\
 & w_2(1) = w_3(0) =  v_2(1)  , \quad
\gb_2 w_2'(1) - \gb_3 w_3'(0) + u_2(1)=0. 
\end{align*}
Since $\gb_1,\gb_2,\gb_3>0$, it follows easily that $x=0$. Hence $A$ is injective, and therefore invertible.
By choosing $\ep>0$ in the previous part of the argument to be sufficiently small we may therefore deduce that $i\RR\subseteq\rho(A)$ with $\|(is-A)\inv\|\lesssim 1+|s|^{1/2}$ for $s\in\RR$.  Thus the semigroup $\T$ is strongly stable by an application of the theorem of Arendt, Batty, Lyubich and V\~{u}~\cite[Thm.~V.2.21]{EngNag00book}, and by~\cite[Thm.~2.4]{BorTom10} we obtain $\|T(t)z_0\|=o(t^{-2})$ as $t\to\infty$ for all $z_0\in D(A)$. Since the energy of a solution with initial data $z_0\in X$ is given by $E(t)=\frac12\|T(t)z_0\|^2$ for all $t\ge0$, the conclusions of Theorem~\ref{thm:WaveMultiHeat} follow immediately.
\end{proof}

\begin{rem}
\label{rem:AltBC}
We may  alternatively consider the wave-heat system with the Neumann boundary condition $y_x^1(t,0)=0$ for $t>0$ at the left end of the first wave equation in Figure~\ref{fig:WHnetworkMultiHeat}. The analysis is largely the same as in our case,  with the important difference that  the point $0$ becomes 
an eigenvalue
 of the operator $A$.
\end{rem}

\bibliographystyle{plain}

\begin{thebibliography}{10}

\bibitem{AalMal13}
A.~Aalto and J.~Malinen.
\newblock Compositions of passive boundary control systems.
\newblock {\em Math. Control. Relat. Fields}, 3(1):1--19, 2013.

\bibitem{Amm07}
K.~Ammari.
\newblock Asymptotic behavior of some elastic planar networks of
  {B}ernoulli--{E}uler beams.
\newblock {\em Appl. Anal.}, 86(12):1529--1548, 2007.

\bibitem{AmmJel04}
K.~Ammari and M.~Jellouli.
\newblock Stabilization of star-shaped networks of strings.
\newblock {\em Differential Integral Equations}, 17(11-12):1395--1410, 2004.

\bibitem{AmmShe18}
K.~Ammari and F.~Shel.
\newblock Stability of a tree-shaped network of strings and beams.
\newblock {\em Math. Methods Appl. Sci.}, 41(17):7915--7935, 2018.

\bibitem{AvaLas16}
G.~Avalos, I.~Lasiecka, and R.~Triggiani.
\newblock Heat-wave interaction in 2--3 dimensions: optimal rational decay
  rate.
\newblock {\em J. Math. Anal. Appl.}, 437(2):782--815, 2016.

\bibitem{BatPau16}
C.J.K. Batty, L.~Paunonen, and D.~Seifert.
\newblock Optimal energy decay in a one-dimensional coupled wave-heat system.
\newblock {\em J. Evol. Equ.}, 16(3):649--664, 2016.

\bibitem{BatPau19}
C.J.K. {Batty}, L.~{Paunonen}, and D.~{Seifert}.
\newblock Optimal energy decay for the wave-heat system on a rectangular
  domain.
\newblock {\em SIAM J. Math. Anal}, 51(2):808--819, 2019.

\bibitem{BorTom10}
A.~Borichev and Yu. Tomilov.
\newblock Optimal polynomial decay of functions and operator semigroups.
\newblock {\em Math. Ann.}, 347(2):455--478, 2010.

\bibitem{CoxZua95}
S.~Cox and E.~Zuazua.
\newblock The rate at which energy decays in a string damped at one end.
\newblock {\em Indiana Univ. Math. J.}, 44(2):545--573, 1995.

\bibitem{DagZua06book}
R.~D\'{a}ger and E.~Zuazua.
\newblock {\em Wave propagation, observation and control in {$1\text{-}d$}
  flexible multi-structures}, volume~50 of {\em Math\'{e}matiques \&
  Applications}.
\newblock Springer-Verlag, Berlin, 2006.

\bibitem{DelPau23}
F.~Dell'Oro, L.~Paunonen, and D.~Seifert.
\newblock Optimal decay for a wave-heat system with {C}oleman--{G}urtin thermal
  law.
\newblock {\em J. Math. Anal. Appl.}, 518(2):Paper No. 126706, 28 pp., 2023.

\bibitem{Duy07}
T.~Duyckaerts.
\newblock Optimal decay rates of the energy of a hyperbolic-parabolic system
  coupled by an interface.
\newblock {\em Asymptot. Anal.}, 51(1):17--45, 2007.

\bibitem{EggKug18}
H.~Egger and T.~Kugler.
\newblock Damped wave systems on networks: exponential stability and uniform
  approximations.
\newblock {\em Numer. Math.}, 138(4):839--867, 2018.

\bibitem{EngNag00book}
K.-J. Engel and R.~Nagel.
\newblock {\em One-Parameter Semigroups for Linear Evolution Equations}.
\newblock Springer-Verlag, New York, 2000.

\bibitem{HanZua16}
Z.-J. Han and E.~Zuazua.
\newblock Decay rates for 1-$d$ heat-wave planar networks.
\newblock {\em Netw. Heterog. Media}, 11(4):655--692, 2016.

\bibitem{JacKai19}
B.~Jacob and J.T. Kaiser.
\newblock Well-posedness of systems of 1-{D} hyperbolic partial differential
  equations.
\newblock {\em J. Evol. Equ.}, 19(1):91--109, 2019.

\bibitem{JacZwa12book}
B.~Jacob and H.~Zwart.
\newblock {\em Linear Port-Hamiltonian Systems on Infinite-Dimensional Spaces},
  volume 223 of {\em Operator Theory: Advances and Applications}.
\newblock Birkh\"{a}user, Basel, 2012.

\bibitem{Klo12}
B.~Kl\"oss.
\newblock The flow approach for waves in networks.
\newblock {\em Oper. Matrices}, 6(1):107--128, 2012.

\bibitem{KraMug21}
M.~Kramar~Fijav\v{z}, D.~Mugnolo, and S.~Nicaise.
\newblock Linear hyperbolic systems on networks: well-posedness and qualitative
  properties.
\newblock {\em ESAIM Control Optim. Calc. Var.}, 27:Paper No. 7, 46 pp., 2021.

\bibitem{LiWan24}
Y.-F. Li and Y.~Wang.
\newblock Sharp polynomial decay rate of multi-link hyperbolic-parabolic
  systems with different network configurations.
\newblock {\em Evol. Equ. Control Theory}, 13(1):12--25, 2024.

\bibitem{MalSta07}
J.~Malinen and O.J. Staffans.
\newblock Impedance passive and conservative boundary control systems.
\newblock {\em Complex Anal. Oper. Theory}, 1(2):279--300, 2007.

\bibitem{MatSik07}
T.~M\'{a}trai and E.~Sikolya.
\newblock Asymptotic behavior of flows in networks.
\newblock {\em Forum Math.}, 19(3):429--461, 2007.

\bibitem{Ng20b}
A.C.S. Ng.
\newblock Optimal energy decay in a one-dimensional wave-heat-wave system.
\newblock In {\em Semigroups of operators---theory and applications}, volume
  325 of {\em Springer Proc. Math. Stat.}, pages 293--314. Springer, Cham,
  2020.

\bibitem{NiPaSe24}
S.~Nicaise, L.~Paunonen, and D.~Seifert.
\newblock Stability of abstract coupled systems.
\newblock Submitted, preprint available at https://arxiv.org/abs/2403.15253,
  2024.

\bibitem{TucWei09book}
M.~Tucsnak and G.~Weiss.
\newblock {\em Observation and Control for Operator Semigroups}.
\newblock Birkh\"auser, Basel, 2009.

\bibitem{ValZua09}
J.~Valein and E.~Zuazua.
\newblock Stabilization of the wave equation on 1-{D} networks.
\newblock {\em SIAM J. Control Optim.}, 48(4):2771--2797, 2009.

\bibitem{ZhaZua04}
X.~Zhang and E.~Zuazua.
\newblock Polynomial decay and control of a 1-$d$ hyperbolic-parabolic coupled
  system.
\newblock {\em J. Differ. Equ.}, 204(2):380--438, 2004.

\bibitem{ZhaZua07}
X.~Zhang and E.~Zuazua.
\newblock Long-time behavior of a coupled heat-wave system arising in
  fluid-structure interaction.
\newblock {\em Arch. Ration. Mech. Anal.}, 184(1):49--120, 2007.

\end{thebibliography}

\end{document}